\documentclass{article}
\usepackage{array,multirow}
\usepackage{amsfonts}
\usepackage{amsmath}
\usepackage{amssymb}
\usepackage{amsthm}
\usepackage{enumerate}
\usepackage[latin1]{inputenc} 
\usepackage{color}
\usepackage{mathtools}
\usepackage{mathdots}
\usepackage{amsmath,amsthm,amssymb}
\usepackage{cite}
\usepackage{authblk}
\usepackage{booktabs}
\usepackage{mathabx}
\usepackage{multirow}
\usepackage{multicol}
\usepackage{tikz}
\usepackage{subdepth}
\usepackage{enumerate}
\usepackage{mathtools}

\usepackage{hyperref}
\newtheorem{theorem}{Theorem}%[section]

\newtheorem{lemma}[theorem]{Lemma}
\theoremstyle{definition}

\newtheorem{remark}[theorem]{Remark}
\newtheorem{definition}[theorem]{Definition}

\newcommand{\CC}{{\mathbb C}}

\newcommand{\rank}{{\mbox{\rm rank\,}}}

\newcommand{\colb}{\color{blue}}

\newcommand{\odd}{\text{\tiny $\mathcal{O}$}}

\newenvironment{smat}{\left[\begin{smallmatrix}}{\end{smallmatrix}\right]}

\usepackage[paperwidth=210mm,paperheight=297mm,textwidth=180mm,textheight=255mm,top=15mm,lmargin=15mm]{geometry}

\title{$X^\top AX=B$}
\author{Alberto Borobia\thanks{Facultad de Ciencias, Universidad Nacional de Educaci\'on a Distancia (UNED), email: {\tt aborobia@mat.uned.es}}, Roberto Canogar\thanks{Facultad de Ciencias, Universidad Nacional de Educaci\'on a Distancia (UNED), email: {\tt rcanogar@mat.uned.es}}, and Fernando De Ter\'{a}n\thanks{Departamento de Matem\'{a}ticas, Universidad Carlos III de Madrid, email: {\tt fteran@math.uc3m.es}}}
\date{\today}

\title{On the consistency of the matrix equation $X^\top A X=B$ when $B$ is symmetric: the case where CFC($A$) includes skew-symmetric blocks}

\begin{document}

\maketitle

\begin{abstract}
  In this paper, which is a follow-up to [A. Borobia, R. Canogar, F. De Ter\'an, {\em Mediterr. J. Math}. 18, 40 (2021)], we provide a necessary and sufficient condition for the matrix equation $X^\top AX=B$ to be consistent when $B$ is symmetric. The condition depends on the canonical form for congruence of the matrix $A$, and is proved to be necessary for all matrices $A$, and sufficient for most of them. This result improves the main one in the previous paper, since the condition is stronger than the one in that reference, and the sufficiency is guaranteed for a larger set of matrices (namely, those whose canonical form for congruence, CFC($A$), includes skew-symmetric blocks).
\end{abstract}

\medskip

{\bf Keywords:} Matrix equation, transpose, congruence, $\top$-Riccati equation, Canonical Form for Congruence, symmetric matrix, bilinear form.

\medskip

{\bf Mathematics subject classification MSC2020:} 15A21, 15A24, 15A63. 

\section{Introduction}

Let $A\in\mathbb C^{n\times n}$ and let $B\in\mathbb C^{m\times m}$  be a symmetric matrix. We are interested in the consistency of the matrix equation 
\begin{equation}\label{maineq}
    X^\top AX=B,
\end{equation}
where $(\cdot)^\top$ denotes the transpose. To be more precise, we want to obtain necessary and sufficient conditions for \eqref{maineq} to be consistent. The main tool to get these conditions is the {\em canonical form for congruence}, {CFC} (see Theorem \ref{cfc_th}), because \eqref{maineq} is consistent if and only if the equation that we obtain after replacing the matrices $A$ and/or  $B$ by their CFCs is consistent. The CFC is a direct sum of three kinds of blocks of different sizes, named Type-0, Type-I, and Type-II, and the idea is to take advantage of this structure to analyze Eq. \eqref{maineq}. In particular, the only symmetric canonical blocks are $I_1=[1]$ and $0_1=[0]$, so the CFC of the symmetric matrix $B$ is of the form CFC$(B)=I_m\oplus 0_k$ (where $I_m$ and $0_k$ are, respectively, a direct sum of $m$ and $k$ copies of $I_1$ and $0_1$). With the help of  Lemma \ref{basic_laws_lemma} we can get rid of the null block $0_k$, so the equation we are interested in is
\begin{equation}\label{maineqbis}
X^\top AX=I_m
\end{equation}
with $m\geq 1$.

In \cite{bcd} we introduced $\tau(A)$, a quantity that  depends on the number of certain Type-0, Type-I, and Type-II blocks appearing in the CFC of $A$, and we  proved in \cite[Th. 2]{bcd} that if  Eq. \eqref{maineqbis} is consistent then $m\leq\tau(A)$. Moreover, the main result of that paper,  \cite[Th. 8]{bcd}, establishes that if the CFC of $A$ contains neither  $H_2(-1)$ nor $H_4(1)$ blocks (which are specific Type-II blocks) then Eq. \eqref{maineqbis} is consistent if and only if  $m\leq\tau(A)$. This is not  necessarily true if we allow the CFC of $A$ to contain blocks $H_2(-1)$ and/or $H_4(1)$ (for instance, it is not true for $A=H_2(-1)$ nor $A=H_4(1)$). 

In the present work  we  introduce a new quantity $\upsilon(A)$, that depends also on the number of certain Type-0, Type-I, and Type-II blocks appearing in the CFC of $A$. In Theorem \ref{necessary_th} we will prove that if Eq. \eqref{maineqbis} is consistent  then  $m\leq\min\{\tau(A),\upsilon(A)\}$. Moreover, according to the main result in the present work (Theorem \ref{main_th}), if  the CFC of $A$ does not contain $H_4(1)$ blocks, then Eq. \eqref{maineqbis} is consistent if and only if $m\leq\min\{\tau(A),\upsilon(A)\}$. However, this is not necessarily true if the CFC contains blocks $H_4(1)$ (it is not true,  for instance, for   $A=H_4(1)$). 

Note that the main result of this paper improves the main one in \cite{bcd} in two senses: (i) the condition here is stronger than the one there; and (ii) the characterization is guaranteed for a larger set of matrices.

In the title we have referred to ``the case where CFC($A$) includes skew-symmetric blocks". This highlights the fact that, compared to \cite{bcd}, in the present work the main result is applied to matrices whose CFC  contains $H_2(-1)$ blocks, which are the only nonzero skew-symmetric blocks in a CFC.

The interest on Eq. \eqref{maineq} goes back to, at least, the 1920's \cite{wedderburn1921}, and it has been mainly devoted to describing the solution, $X$, for matrices $A,B$ over finite fields and when $A$ and/or $B$ have some specific structure \cite{buckhiester1,buckhiester2,buckhiester3,carlitz,fulton79,hodges,wei-zhang}. More recently, some related equations have been analyzed \cite{ikramov14} and, in particular, in connection with applications \cite{bimp,benner-palitta,benzi-viviani}. In \cite{bcd-skew} we have addressed the consistency of Eq. \eqref{maineq} when $B$ is skew-symmetric, where it is emphasized the connection between the consistency of \eqref{maineq} and the dimension of the largest subspace of $\CC^n$ for which the bilinear form represented by $A$ is skew-symmetric and non-degenerate. The same connection holds after replacing skew-symmetric by symmetric, which is the structure considered in the present work. 

The paper is organized as follows. In Section \ref{framework_sec} we introduce the basic notation and definitions (like the CFC), and we also recall some basic results that are used later. In Section \ref{rho_sigma_sec} the quantities $\tau(A)$ and $\upsilon(A)$ are introduced. Section \ref{necessary_sec} presents the necessary condition  for Eq. \eqref{maineqbis} to be consistent (Theorem \ref{necessary_th}), whereas in Section \ref{main_sec}  we show that when the CFC of $A$ does not contain blocks $H_4(1)$ this condition is sufficient as well (Theorem \ref{main_th}). In between these two sections, Section \ref{absorb_sec} is devoted to introduce the tools (by means of several technical lemmas) that are used to prove the sufficiency of the condition. Finally, in Section \ref{conclusion_sec} we summarize the main contributions of this work and indicate the main related open question.

\section{Basic approach and definitions}\label{framework_sec}

Throughout the manuscript, $I_n$ and $0_n$ denote, respectively, the identity and the null matrix with size $n\times n$. By $0_{m\times n}$ we denote the null matrix of size $m\times n$. By $\mathfrak i$ we denote the imaginary unit (namely, $\mathfrak i^2=-1$), and by $e_j$ we denote the $j$th canonical vector (namely, the $j$th column of the identity matrix) of the appropriate size. The notation $M^{\oplus k}$ stands for a direct sum of $k$ copies of the matrix $M$.

Following the approach in \cite{bcd} and \cite{bcd-skew}, a key tool in our developments is the {\em canonical form for congruence} (CFC). For the ease of reading we first recall the CFC, that depends on the following matrices:
\begin{itemize}
    \item $
J_k(\mu):=\left[\begin{array}{c@{\mskip8mu}c@{\mskip8mu}c@{\mskip8mu}c}
\mu&1\\[-4pt]&\ddots&\ddots\\[-4pt]&&\mu&1\\[-2pt]&&&\mu\end{array}\right]
$ is a $k\times k$ {\em Jordan block associated with
$\mu\in\CC$};
\item $
\Gamma_k:=\left[\begin{array}{c@{\mskip8mu}c@{\mskip8mu}c@{\mskip8mu}c@{\mskip8mu}c@{\mskip8mu}c}0&&&&&(-1)^{k+1}\\[-4pt]
&&&&\iddots&(-1)^k\\[-4pt]&&&-1&\iddots&\\&&1&1&&\\&-1&-1&&&\\1&1&&&&0\end{array}\right]_{k\times k}$
\qquad for $k\geq1$ (note that $\Gamma_1=I_1=[1]$); and
\item $
    H_{2k}(\mu):=\begin{bmatrix}
    0&I_k\\
     J_k(\mu)&0
    \end{bmatrix},
$ for $k\geq1$,
where $J_k(\mu)$ is a $k\times k$ Jordan block associated with $\mu\in\CC$.
\end{itemize}

\begin{theorem}\label{cfc_th}{\rm (Canonical form for congruence, CFC) \cite[Th. 1.1]{hs2006}.}
Each square complex matrix is congruent to a direct sum, uniquely determined up to permutation of addends, of canonical matrices of the following three types
\begin{center}
\begin{tabular}{|c|c|}\hline
     Type 0& $J_k(0)$ \\\hline
     Type I& $\Gamma_k$\\\hline
     Type II&\begin{tabular}{c}$H_{2k}(\mu)$,\\ \footnotesize$0 \neq\mu\neq(-1)^{k+1}$\\
\footnotesize($\mu$ is determined up to replacement by $\mu^{-1}$)\end{tabular}\\\hline
\end{tabular}
\end{center}
\end{theorem}

Following \cite{bcd-skew}, the notation $A \rightsquigarrow B$ means that the equation $X^\top A X=B$ is consistent, and $A\overset{X_0}{\rightsquigarrow} B$ means that $X_0^\top AX_0=B$. The following result, that was presented in \cite[Lemma 4]{bcd-skew}, includes some basic laws of consistency that are straightforward to check. 

\begin{lemma} \label{basic_laws_lemma}
\label{directsum.lemma} \label{TransitivityOfConsistency_lem} \label{suma_directa_rightsquigarrows_lem} 
{\rm({\bf Laws of consistency})}. For any complex square matrices $A,B,C,A_i,B_i$, the following properties hold:
\begin{enumerate}[{\rm(i)}]
    \item {\bf Addition law.} \label{additive_law}
    If $A_i \overset{X_i}{\rightsquigarrow} B_i$, for $1\leq i\leq k$, then  $\bigoplus_{i=1}^{k} A_i \overset{X}{\rightsquigarrow} \bigoplus_{i=1}^{k} B_i$, with $X=\bigoplus_{i=1}^{k} X_i$.
    
    \item {\bf Transitivity law.}  \label{transitive_law}
    If $A \overset{X_0}{\rightsquigarrow} B$ and $B \overset{Y_0}{\rightsquigarrow} C$, then $A \overset{X_0Y_0}{\rightsquigarrow} C$.  
    
\item {\bf Permutation law.} \label{perm_law} $\bigoplus_{i=1}^\ell A_i\rightsquigarrow \bigoplus_{i=1}^\ell A_{\sigma(i)}$, for any permutation $\sigma$ of $\{1,\hdots,\ell\}$.
    
    \item {\bf Elimination law.}\label{elim_law} 
    $A \oplus B \overset{X_0}{\rightsquigarrow} A$, with $X_0=\left[\begin{smallmatrix}I_n\\ 0\end{smallmatrix}\right]$, and where $n$ is the size of $A$.

    \item{\bf Canonical reduction law.}\label{reduction_law}  If $A$ and $B$ are congruent to, respectively, $\widetilde A$ and $\widetilde B$, then $A\rightsquigarrow B$ if and only if $\widetilde A\rightsquigarrow\widetilde B$.
    
    \item {\bf $J_1(0)$-law.}  \label{zeroes_law}
    For $k,\ell\geq 0$ we have  $A\oplus J_1(0)^{\oplus k} \rightsquigarrow  B \oplus J_1(0)^{\oplus \ell}$ if and only if  $A \rightsquigarrow  B$.

\end{enumerate}
\end{lemma}

By the Canonical reduction law, in Eq.~\eqref{maineq} we will assume without loss of generality that $A$ and $B$ are given in CFC.

When $B$ is symmetric, the CFC of $B$ is $I_{m_1}\oplus 0_{m_2}$. Then, as a consequence of the Canonical reduction law, we may restrict ourselves to the case where the right-hand side of  \eqref{maineq} is of this form. Moreover, as a consequence of the $J_1(0)$-law, in our developments we will consider $B=I_m$ in Eq.~\eqref{maineq} (leading to Eq.~\eqref{maineqbis}). Therefore, our goal is to characterize those matrices $A$ such that $A\rightsquigarrow I_m$, for a fixed $m\geq1$. This will be done by concatenating several equations $A\rightsquigarrow A_1\rightsquigarrow\cdots\rightsquigarrow A_k\rightsquigarrow I_m$, since the Transitivity law allows us to conclude that $A\rightsquigarrow I_m$. For this reason, we will use the word ``transformation" for a single equation $A\rightsquigarrow B$.

One way to determine the CFC of an invertible matrix $A$ is by means of its {\em cosquare}, $A^{-\top}A$ (see \cite{hs2006}), where $(\cdot)^{-\top}$ denotes the transpose of the inverse. Moreover, the cosquare will be used to determine whether two given invertible matrices are congruent, using the following result.

\begin{lemma}\label{cosquares_lem}\emph{(\cite[Lemma 2.1]{hs2006})}.
Two invertible matrices are congruent if and only if their cosquares are similar.
\end{lemma}

\subsection{The matrices $\widetilde\Gamma_k$ and $\widetilde H_{2k}(\mu)$} 

Instead of the blocks $\Gamma_k$ and $H_{2k}(\mu)$ we will use the following blocks, for $k\geq 1$:
\begin{equation*}
    \widetilde\Gamma_k:=\begin{bmatrix}
     1&1\\-1&0&1\\&1&0&1\\&&\ddots&\ddots&\ddots\\&&&(-1)^k&0&1\\&&&&(-1)^{k+1}&0
    \end{bmatrix}_{k\times k} \qquad \text{and} \qquad
 \widetilde H_{2k}(\mu) :=\left[   \begin{array}{ccccccc}
 0   & 1 &     &   &     &   & \\
 \mu & 0 & 1   &   &     & \mathbf{0}  & \\
     & 0 & 0   & 1 &     &   & \\
     &   & \mu & 0 & 1   &   & \\
     &   &     & 0 & \ddots   & \ddots & \\
     & \mathbf{0}  &     &   & \ddots & 0 & 1\\
     &   &     &   &    & \mu &0\\
\end{array} 
\right]_{2k\times2k}.
\end{equation*}

We claim that $\widetilde\Gamma_k$ and $\widetilde H_{2k}(\mu)$ are congruent to, respectively, $\Gamma_k$ and $H_{2k}(\mu)$. 

In order to prove that $\Gamma_k$ and $\widetilde\Gamma_k$ are congruent, we give an indirect proof. Two matrix pairs $(A,B)$ and $(A',B')$ are {\em strictly equivalent} if there are invertible matrices $R$ and $S$ such that $R A S = A'$ and $R B S = B'$. It is known (see, for instance, \cite[Lemma 1]{semaj}) that two matrices $A,B\in\CC^{n\times n}$ are congruent if and only if $(A,A^\top)$  and  $(B,B^\top)$ are strictly equivalent. Since $(\Gamma_k,\Gamma_k^\top)$ and $\big(J_k\big((-1)^{k+1}\big),I_k\big)$ are strictly equivalent (see \cite[Th. 4]{semaj}) and $\big(J_k\big((-1)^{k+1}\big),I_k\big)$ and $(\widetilde\Gamma_k,\widetilde\Gamma_k^\top)$ are strictly equivalent as well (see Eq. (5) in \cite{fhs08}), the pairs $(\Gamma_k,\Gamma_k^\top)$ and $(\widetilde\Gamma_k,\widetilde\Gamma_k^\top)$ are strictly equivalent, so $\Gamma_k$ and $\widetilde\Gamma_k$ are congruent. Another alternative to show that $\Gamma_k$ and $\widetilde\Gamma_k$ are congruent is by checking that their cosquares are similar to $J_k((-1)^{k+1})$ and then using Lemma \ref{cosquares_lem}.

To see that $H_{2k}(\mu)$ and $\widetilde H_{2k}(\mu)$ are congruent, consider the permutation matrix
\begin{align*}
P_{2k}=\begin{bmatrix}
e_1 & e_{k+1} & e_2 & e_{k+2} & \cdots & e_k & e_{2k}
\end{bmatrix},
\end{align*}
and note that
\[
\widetilde H_{2k}(\mu)=P_{2k}^\top H_{2k}(\mu) P_{2k}=
P_{2k}^\top 
\begin{bmatrix}  0&I_k\\ J_k(\mu)&0 \end{bmatrix} P_{2k}.\]
Therefore, the congruence by $P_{2k}$ is actually a simultaneous permutation of rows and columns of $H_{2k}(\mu)$. More precisely, we start with 
$\begin{smat}  0&I_k\\ J_k(\mu)&0 \end{smat}$ and
move rows (and columns) $(k+1, k+2, \ldots, 2k)$ to, respectively, rows (and columns) $(2,4, \ldots, 2k)$; and we also move rows (and columns) $(1, 2, \ldots, k)$ to rows (and columns) $(1,3, \ldots, 2k-1)$, respectively. So the $1$'s coming from the block $I_k$ and the $1$'s coming from the superdiagonal of the block $J_k(\mu)$ in $H_{2k}(\mu)$, get shuffled to form the superdiagonal of $P_{2k}^\top H_{2k}(\mu) P_{2k}$. Moreover, the $\mu$'s from the block $J_k(\mu)$ in $H_{2k}(\mu)$ are taken to the positions $(2,1),(4,3),\hdots,(2k,2k-1)$ in $\widetilde H_{2k}(\mu)$. 

The advantage in using the matrices $\widetilde\Gamma_k$ and $\widetilde H_{2k}(\mu)$ instead of, respectively, $\Gamma_k$ and $H_{2k}(\mu)$, is that the first ones are tridiagonal, and this structure is more convenient for our proofs. Tridiagonal canonical blocks have been already used in \cite{fhs08} (actually, $\widetilde\Gamma_k$ is exactly the one introduced in Eq. (3)  for $\varepsilon=1$ in that reference).

For the rest of the manuscript, we will replace the blocks $\Gamma_k$ by $\widetilde\Gamma_k$ and $H_{2k}(\mu)$ by $\widetilde H_{2k}(\mu)$, so, in particular, we will assume that the CFC is a direct sum of blocks $J_k(0)$, $\widetilde\Gamma_k$, and $\widetilde H_{2k}(\mu)$. The only exceptions to this rule are $\Gamma_1$ which is equal to $\widetilde\Gamma_1$, and $H_2(-1)$ which is equal to $\widetilde H_2(-1)$.

\section{The quantities $\tau(A)$ and $\upsilon(A)$}\label{rho_sigma_sec}

The main result of this work (Theorem~\ref{main_th}) depends on two intrinsic quantities of the matrix A, that we denote by $\tau(A)$ and $\upsilon(A)$.
In this section, we introduce them and present some basic properties that will be used later.

\begin{definition} \label{components_CFC_def}
Let $A$ be  a complex $n\times n$ matrix  and consider its {\rm CFC}, where  
\begin{enumerate}[\rm(i)]
\item\label{type01} $j_1$  is the number of  Type-$0$ blocks with size $1$;

\item  $j_{\odd}$  
is the number of Type-$0$ blocks with odd size at least 3;

\item  $\gamma_{\odd}$ 
is the number of  Type-I blocks with odd size;

\item  $\gamma_{\varepsilon}$ is the number of Type-I blocks with even size;

\item  $h^-_{2\odd}$ 
is the number of Type-II blocks $\widetilde H_{4k-2}(-1)$ for any $k\geq1$; and

\item  $h^+_{2\varepsilon}$ is the number of Type-II blocks $\widetilde H_{4k}(1)$ for any $k\geq1$;

\item\label{othertypes} it has an arbitrary number of other Type-$0$ and Type-II blocks.
\end{enumerate}
Then we define the quantities
\begin{equation} \label{rho(A)_formula}
    \tau(A):=\frac{n-j_1+j_{\odd}+\gamma_{\odd}+2 h^+_{2\varepsilon}}{2}
     \ \ \text{ and } \ \upsilon(A):= n-j_1-j_{\odd}-\gamma_{\varepsilon}-2h^-_{2\odd}.
\end{equation}
\end{definition}

The quantities $\tau$ and $\upsilon$ satisfy the following essential additive  properties (the proof is straightforward):
\begin{equation} \label{additive}\tau(A_1\oplus\cdots \oplus A_k)=\tau(A_1)+\cdots+\tau(A_k)\quad \text{and} \quad   \upsilon(A_1\oplus\cdots \oplus A_k)=\upsilon(A_1)+\cdots+\upsilon(A_k).
\end{equation}

The notation for the quantities in Definition \ref{components_CFC_def} follows the one in \cite{bcd-skew}. In particular, the letters used for the number of blocks in parts (i)--(vi) resemble the notation for the corresponding blocks (see \cite[Rem. 6]{bcd-skew}).
In \cite{bcd} we had not yet adopted this notation. The correspondence between the notation in that paper and the one used here is the following:  $d \rightarrow j_1, r\rightarrow j_\odd, s\rightarrow \gamma_\odd, t\rightarrow h^+_{2\varepsilon}$. The  values $\gamma_\varepsilon$ and $h^-_{2\odd}$ played no role in \cite{bcd}.

 Table \ref{taupsilon_table} contains the values of $\tau(A)$ and $\upsilon(A)$ for $A$ being a single canonical block in the CFC. We have displayed the values in three categories, from top to bottom, namely: first, those with $\tau(A)=\upsilon(A)$; second, those for which $\tau(A)<\upsilon(A)$; and, finally, those with $\tau(A)>\upsilon(A)$.

\begin{table}[]
    $$
\begin{array}{|l|c|c|c|}\hline
     A&{\rm conditions}&\tau(A) & \upsilon(A)\\\hline\hline
     J_1(0)&-& 0&0\\
     J_3(0)&-&2&2\\
     \Gamma_1&-&1&1\\
     \widetilde\Gamma_2&-&1&1\\\hline
     J_{2k+1}(0)&k\geq 2&k+1&2k\\
     J_{2k}(0)&k\geq 1&k&2k\\
     \widetilde\Gamma_{2k+1}&k\geq 1&k+1&2k+1\\
     \widetilde\Gamma_{2k}&k\geq 2&k&2k-1\\
     \widetilde H_{4k-2}(-1)&k\geq 2&2k-1&4k-4\\
     \widetilde H_{4k}(1)&k\geq 1&2k+1&4k\\
     \widetilde H_{2k}(\mu)&k\geq1,\mu\neq0,\pm1&k&2k\\\hline
     H_2(-1)&-&1&0\\\hline
\end{array}
$$
    \caption{\footnotesize Values of $\tau$ and $\upsilon$ for any single canonical block.}
    \label{taupsilon_table}
\end{table}

Notice that $\tau(A)\leq \upsilon (A)$ whenever the CFC of $A$ consists of just a single canonical block, except for $H_2(-1)$. This, together with \eqref{additive}, implies the following result.

\begin{lemma}\label{rho<=sigma_lem}
If the {\rm CFC} of $A$ has no blocks of type $H_2(-1)$ then $\tau(A)\leq \upsilon(A)$.
\end{lemma}

In order for the condition that we obtain (in Theorem \ref{necessary_th}) to be sufficient, the following notion is key.

\begin{definition}\label{invariant_def}
The transformation $A\rightsquigarrow B$ is {\em $(\tau,\upsilon)-$invariant} if the following three conditions are satisfied:
\begin{itemize}
    \item $X^\top AX=B$ is consistent,
    \item $\tau(A)=\tau(B)$, and 
    \item $\upsilon(A)=\upsilon(B)$.
\end{itemize}
\end{definition}

\section{A necessary condition}\label{necessary_sec}

In this section, we introduce a necessary condition on the matrix $A$ for $A\rightsquigarrow I_m$ (namely, for Eq.~\eqref{maineq} to be consistent when $B$ is symmetric and invertible). This condition improves the one provided in \cite[Th. 2]{bcd}, namely $m\leq\tau(A)$.
 
 \begin{theorem}\label{necessary_th}
If $A$ is a complex square matrix such that $X^\top AX= I_m$ is consistent, then $m \leq  \min\{\tau(A),\upsilon(A)\}$.
\end{theorem}

\begin{proof} 
In \cite[Th. 2]{bcd} it was proved that $m\leq\tau(A)$ (though the notation $\tau$ was not used there).
Let us see that $m\leq \upsilon(A)$ as well. Assuming that the CFC of $A$ is as in Definition \ref{components_CFC_def}, in the proof of Theorem 8 of \cite{bcd-skew} it was  showed that
\begin{equation}\label{n-rank(A+AT)}
    n-\rank(A+A^\top)=j_\odd+\gamma_\varepsilon+2h_{2\odd}^-.
\end{equation}
By hypothesis, there exists some $X_0 \in\CC^{n\times m}$ such that     $X_0^\top A X_0=I_m$. Now, transposing this equation and adding it up, we get $X_0^\top(A+A^\top)X_0= 2I_m$. From this identity, and using \eqref{n-rank(A+AT)}, we obtain
$$
m=\rank(X_0^\top(A+A^\top)X_0)\leq \rank(A+A^\top) =n-j_\odd-\gamma_\varepsilon-2h_{2\odd}^-,
$$
so $m\leq n-j_\odd-\gamma_\varepsilon-2h_{2\odd}^- =\upsilon(A),$ as claimed.
\end{proof}

\section{Absorbing the  $H_2(-1)$ blocks}\label{absorb_sec}

The main goal in the rest of the manuscript is to prove that the necessary condition presented in Theorem \ref{necessary_th} is also sufficient when the CFC of $A$ does not contain $\widetilde H_4(1)$ blocks. If the CFC of $A$ contains neither  $H_2(-1)$ nor $\widetilde H_4(1)$ blocks, this is already known \cite[Th. 8]{bcd}. In that case, as a consequence of Lemma \ref{rho<=sigma_lem}, the condition for $A\rightsquigarrow I_m$ reduces to $m \leq \tau(A)$. When the CFC of $A$ does not contain blocks $\widetilde H_4(1)$ but contains blocks $H_2(-1)$, this is no longer true (see, for instance, Example 1 in \cite{bcd}), and then the quantity $\upsilon(A)$ comes into play. This is an indication that the presence of blocks $H_2(-1)$ in the CFC of $A$ deserves a particular treatment. In this section, we show how to deal with this type of blocks. To be more precise, we see that some blocks $H_2(-1)$ can be  combined with other type of blocks in order to ``eliminate" them by means of a $(\tau,\upsilon)-$invariant transformation. In this case, we say that the block $H_2(-1)$ has been ``absorbed". We will consider separately the cases of Type-0, Type-I, and Type-II blocks, in Sections \ref{type0_sec}, \ref{typeI_sec}, and \ref{type-II_sec}, respectively.

The following notation is used in the proofs of this section: $E_{\alpha\times\beta}$ denotes the $\alpha\times\beta$ matrix whose $(\alpha,1)$ entry is equal to $1$ and the remaining entries are zero.

\subsection{The case of Type-0 blocks}\label{type0_sec}

In Lemma \ref{joining-Js_lem}, we show how to ``absorb" a block $H_2(-1)$ with a Type-0 block, $J_k(0)$, with $k\neq3$. In the statement, $J_0(0)$ stands for an empty block.

\begin{lemma}\label{joining-Js_lem}
The following  transformation is  ($\tau,\upsilon)-$invariant:

\begin{equation}\label{reducing-type0}
    J_{k}(0) \oplus H_2(-1) {\rightsquigarrow} J_{k-2}(0) \oplus \Gamma_1^{\oplus 2},\qquad  \mbox{for $k=2$ and $k  \geq 4$}.
\end{equation}
\end{lemma}

\begin{proof}
By considering separately the cases where $k$ in \eqref{reducing-type0} is odd ($k=2t+1$) and even ($k=2t$), using \eqref{additive} and looking at Table \ref{taupsilon_table}, we obtain:
$$
\begin{array}{lclclc}
    \tau\big(J_{2t+1}(0) \oplus H_2(-1)\big)&=&t+2&=&\tau\big(J_{2t-1}(0) \oplus \Gamma_1^{\oplus 2}\big),&\mbox{for $t\geq2$},\\ \upsilon\big(J_{2t+1}(0) \oplus H_2(-1)\big)&=&2t&=&\upsilon\big(J_{2t-1}(0) \oplus \Gamma_1^{\oplus 2}\big),&\mbox{for $t\geq2$},\\
    \tau\big(J_{2t}(0) \oplus H_2(-1)\big)&=&t+1&=&\tau\big(J_{2t-2}(0) \oplus \Gamma_1^{\oplus 2}\big),&\mbox{for $t\geq1$},\\ \upsilon\big(J_{2t}(0) \oplus H_2(-1)\big)&=&2t&=&\upsilon\big(J_{2t-2}(0) \oplus \Gamma_1^{\oplus 2}\big),&\mbox{for $t\geq1$},
\end{array}
$$
so both sides of the transformation in (\ref{reducing-type0}) have the same $\tau$ and $\upsilon$. Now let us prove the consistency. 

The result is true for $k=2$, since 
$$
J_{2}(0) \oplus H_2(-1)\overset{X_2}{\rightsquigarrow} \Gamma_1^{\oplus 2}, \quad\text{ for } X_2=\begin{smat} 
 \mathfrak{i} & 1 \\
 -\mathfrak{i} & 1 \\
 0 & 1 \\
 \mathfrak{i} & 0 
\end{smat}.$$ 
Let us prove it for $k\geq 4$.  Note that $
J_{a+b}(0)=\begin{smat}
J_a(0)&E_{a\times b}\\0&  J_b(0)
\end{smat}$. 
If    
$X_3={\scriptsize\left[\begin{array}{c|c}
1 & \begin{matrix}
0 & 0
\end{matrix} \\ \hline
\begin{matrix}
0 \\ -1 \\ -1 \\ 1 
\end{matrix} & \text{\normalsize $X_2$}
\end{array}\right]}$
 then 
\begin{align*}
\left(I_{k-3}\oplus X_3\right)^\top \Big(J_{k}(0) \oplus H_2(-1) \Big) \left(I_{k-3}\oplus X_3\right) &=
\begin{bmatrix}
I_{k-3} & 0 \\
0 & X_3^\top
\end{bmatrix}
\begin{bmatrix}
J_{k-3}(0) & E_{(k-3)\times 5} \\
0 & J_3(0) \oplus H_2(-1)
\end{bmatrix}
\begin{bmatrix}
I_{k-3} & 0 \\
0 & X_3
\end{bmatrix} \\
&= 
\begin{bmatrix}
J_{k-3}(0) & E_{(k-3)\times 5} X_3 \\
0 & X_3^\top \left(J_3(0) \oplus H_2(-1)\right) X_3
\end{bmatrix}\\
&= \begin{bmatrix}
J_{k-3}(0) & E_{(k-3)\times 3} \\
0 & J_1(0)\oplus \Gamma_1^{\oplus 2}
\end{bmatrix} \\ 
&= J_{k-2}(0)\oplus \Gamma_1^{\oplus 2},
\end{align*}
as wanted.
\end{proof}

We will also use the following result, whose proof is straightforward.

\begin{lemma}\label{j3_lem}
The following  transformation is $(\tau,\upsilon)-$invariant:
$$
J_3(0) \overset{X_0}{\rightsquigarrow} \Gamma_1^{\oplus 2},\qquad \mbox{for $X_0= \begin{smat} 1 & 0 \\1 & \mathfrak i \\0 & -\mathfrak i \end{smat}$}.
$$
\end{lemma}

\subsection{The case of Type-I blocks}\label{typeI_sec}
%In our next result, we continue  with the same strategy that we applied to blocks of Type-0 and Type-II. It this case, it consists in to see for any  block of Type-I  how many blocks of type $H_2(-1)$ are necessary so that via a ($\tau,\upsilon)-$invariant transformation we reduce the size of the block and obtain some $\Gamma_1$ block. 

Lemma \ref{joining-Gammas_lem} is the counterpart of Lemma \ref{joining-Js_lem} for Type-I blocks, where $\Gamma_k$ is replaced by $\widetilde\Gamma_k$.

\begin{lemma}\label{joining-Gammas_lem}
The following transformation is $(\tau,\upsilon)-$invariant:
\begin{equation}\label{reducing-typeI}
    \widetilde\Gamma_{k} \oplus H_2(-1)  {\rightsquigarrow} \Gamma_1^{\oplus 2} \oplus \widetilde\Gamma_{k-2},\ \ for \ k\geq 3.
\end{equation}
%\begin{enumerate}[{\rm(i)}]
%\item $\widetilde\Gamma_{3} \oplus H_2(-1)  %\overset{P}{\rightsquigarrow} H_2(-1)\oplus \widetilde\Gamma_{3}\oplus 
%\overset{}{\rightsquigarrow} \Gamma_1^{\oplus 3}$.%, for $P=\begin{smat} 0&I_3\\I_2&0
%\end{smat}$ and $X_3=\begin{smat}
%1 & 0 & \mathfrak{i} \\
% 0 & -\mathfrak{i} & 0  \\
% 0 & 1 & 0 \\
% -\mathfrak{i} & 0 & 1 \\
% \frac{\mathfrak{i}}{2} & 0 & \frac{1}{2} \\
% \end{smat}$. 

%\item $\widetilde\Gamma_{k} \oplus H_2(-1)  %\overset{P}{\rightsquigarrow} H_2(-1)\oplus \widetilde\Gamma_{k}\oplus
%\overset{}{\rightsquigarrow} \Gamma_1^{\oplus 2} \oplus \widetilde\Gamma_{k-2}$, for $k\geq 4$.%, $P=\begin{smat}
%0&I_k\\I_2&0
%\end{smat}$ and $X_k=\begin{smat}
% 1 & 0 & \mathfrak{i} & 0 \\
% 0 & -\mathfrak{i} & 0 & -\frac{\mathfrak{i}}{2} \\
% 0 & 1 & 0 & 0 \\
% -\mathfrak{i} & 0 & 1 & 0 \\
% \frac{\mathfrak{i}}{2} & 0 & \frac{1}{2} & 0 \\
% 0 & 0 & 0 & 1 \\
% \end{smat}\oplus I_{k-4}$.
%\end{enumerate}
\end{lemma}
\begin{proof} Considering again separately the cases where $k$ in \eqref{reducing-typeI} is odd ($k=2t+1$) and even ($k=2t$), using \eqref{additive} and looking at Table \ref{taupsilon_table}, we obtain: 
$$
\begin{array}{lclclc}
\tau\big(\widetilde\Gamma_{2t+1}\oplus H_2(-1)\big)&=&t+2&=&\tau\big(\Gamma_1^{\oplus2}\oplus\widetilde\Gamma_{2t-1}\big)&\mbox{for $t\geq 1$},\\
\upsilon\big(\widetilde\Gamma_{2t+1}\oplus H_2(-1)\big)&=&2t+1&=&\upsilon\big(\Gamma_1^{\oplus2}\oplus\widetilde\Gamma_{2t-1}\big)&\mbox{for $t\geq 1$},\\
\tau\big(\widetilde\Gamma_{2t}\oplus H_2(-1)\big)&=&t+1&=&\tau\big(\Gamma_1^{\oplus2}\oplus\widetilde\Gamma_{2t-2}\big)&\mbox{for $t\geq 2$},\\
\upsilon(\widetilde\Gamma_{2t}\oplus H_2(-1))&=&2t-1&=&\upsilon(\Gamma_1^{\oplus2}\oplus\widetilde\Gamma_{2t-2})&\mbox{for $t\geq 2$.}
\end{array}
$$
so both sides of the transformation in (\ref{reducing-typeI}) have the same $\tau$ and $\upsilon$. Now let us prove the consistency.

For $k=3$ we have
%We {\colb are going to consider separately the cases $k=3$ and $k\geq4$.} We break up both cases in two transformations and prove the following:
 $$\widetilde\Gamma_{3} \oplus H_2(-1)  \overset{P}{\rightsquigarrow} H_2(-1)\oplus \widetilde\Gamma_{3}  \overset{X_3}{\rightsquigarrow} \Gamma_1^{\oplus 3}, \text{ for } P=\begin{smat}
0&I_3\\I_2&0
\end{smat},  \text{ and } X_3=\begin{smat}
1 & 0 & \mathfrak{i} \\
 0 & -\mathfrak{i} & 0  \\
 0 & 1 & 0 \\
 -\mathfrak{i} & 0 & 1 \\
 \frac{\mathfrak{i}}{2} & 0 & \frac{1}{2} \\
 \end{smat},$$
as can be directly checked. For $k\geq4$ we are going to prove that
 $$\widetilde\Gamma_{k} \oplus H_2(-1)  \overset{P}{\rightsquigarrow} H_2(-1)\oplus \widetilde\Gamma_{k}  \overset{X_k}{\rightsquigarrow} \Gamma_1^{\oplus 2} \oplus \widetilde\Gamma_{k-2}, \text{ with $P=\begin{smat}
0&I_k\\I_2&0
\end{smat}$ and $X_k=
{\scriptsize\left[\begin{array}{ccc|c}
  &  &  & 0 \\
  &  &  & -\frac{\mathfrak{i}}{2} \\
  & \text{\normalsize $X_3$} &  & 0 \\
  &  &  & 0 \\
  &  &  & 0 \\ \hline
 0 & 0 & 0 & 1 \\
 \end{array}\right]}\oplus I_{k-4}$,}$$
where the first transformation is  just a block permutation. So for the rest of the proof we will focus on the second transformation. We use the following notation: $A(i:j)$ is the principal submatrix of $A$ containing the rows and columns from the $i$th to the $j$th ones.

If $k=4$ then $H_2(-1)\oplus \widetilde \Gamma_4 \overset{X_4}{\rightsquigarrow} \Gamma_1^{\oplus 2}\oplus \widetilde \Gamma_2$, for $X_4$ as above, as can be directly checked.%=%{\scriptsize\begin{bmatrix}
 %1 & 0 & \mathfrak{i} & 0 \\
 %0 & -\mathfrak{i} & 0 & -\frac{\mathfrak{i}}{2} \\
%0 & 1 & 0 & 0 \\
 %-\mathfrak{i} & 0 & 1 & 0 \\
 %\frac{\mathfrak{i}}{2} & 0 & \frac{1}{2} & 0 \\
 %0 & 0 & 0 & 1 \\
 %\end{bmatrix}}=
% {\scriptsize\left[\begin{array}{ccc|c}
%  &  &  & 0 \\
%  &  &  & -\frac{\mathfrak{i}}{2} \\
%  & \text{\normalsize $X_3$} &  & 0 \\
%  &  &  & 0 \\
%  &  &  & 0 \\ \hline
% 0 & 0 & 0 & 1 \\
% \end{array}\right]}$.

If $k>4$ then $H_2(-1)\oplus \widetilde \Gamma_k \overset{X_k}{\rightsquigarrow} \Gamma_1^{\oplus 2}\oplus \widetilde \Gamma_{k-2}$, for $X_k=X_4\oplus I_{k-4}$. To prove it we will use the identities
%$$
%{\colr \qquad 
%\widetilde\Gamma_t=\begin{smat}
%\widetilde\Gamma_a&E_{a\times(t-a)}\\E_{a\times(t-a)}^\top&\widetilde\Gamma_{t}(a+1: t)
%\end{smat} \qquad \text{for  $1\leq a < t$}}
%$$
$$\widetilde\Gamma_k=\begin{bmatrix}
\widetilde\Gamma_4&E_{4\times(k-4)}\\E_{4\times(k-4)}^\top&\widetilde\Gamma_{k}(5: k)
\end{bmatrix} \quad \text{ and } \quad \widetilde\Gamma_{k-2}=\begin{bmatrix}
\widetilde\Gamma_2&E_{2\times(k-4)}\\E_{2\times(k-4)}^\top&\widetilde\Gamma_{k-2}(3: k-2)
\end{bmatrix}, $$
so that
\begin{align*}
\left(X_4\oplus I_{k-4}\right)^\top \left( H_2(-1) \oplus \widetilde \Gamma_{k}  \right) \left(X_4\oplus I_{k-4}\right) &=
\begin{bmatrix}
X_4^\top & 0 \\
0 & I_{k-4}
\end{bmatrix}
\begin{bmatrix}
H_2(-1) \oplus \widetilde \Gamma_4  & E_{6 \times (k-4)} \\
E_{6 \times (k-4)}^\top & \widetilde\Gamma_{k}(5:{ k})
\end{bmatrix}
\begin{bmatrix}
X_4 & 0 \\
0 & I_{k-4}
\end{bmatrix} \\
&= 
\begin{bmatrix}
X_4^\top \Big(H_2(-1) \oplus \widetilde \Gamma_4\Big) X_4  & X_4^\top E_{6 \times (k-4)} \\
E_{6 \times (k-4)}^\top X_4 & \widetilde\Gamma_{k}(5:k)
\end{bmatrix}\\
&= \begin{bmatrix}
\Gamma_1^{\oplus 2}\oplus \widetilde \Gamma_2  & E_{4 \times (k-4)} \\
E_{4 \times (k-4)}^\top &  \widetilde\Gamma_{k-2}(3: k-2)
\end{bmatrix} \\ 
&= \Gamma_1^{\oplus 2} \oplus \widetilde \Gamma_{k-2}
\end{align*}
where in the last-but-one equality we use that  $\widetilde\Gamma_{k}(5:k)=\widetilde\Gamma_{k-2}(3:k-2)$.

%where, in the first identity, we have used that $\widetilde\Gamma_k=\begin{smat} \widetilde\Gamma_4&E_{4\times(k-4)}\\E_{4\times(k-4)}^\top&\widetilde\Gamma_{k}(5: k) \end{smat}$, and in the last identity  that   $\widetilde\Gamma_{k-2}=\begin{smat} \widetilde\Gamma_2&E_{2\times(k-4)}\\E_{2\times(k-4)}^\top&\widetilde\Gamma_{k}(5: k) \end{smat}$.
\end{proof}

\subsection{The case of Type-II blocks}\label{type-II_sec}

%In our next result, we continue  with the same strategy that we applied to blocks of Type-0. It this case, it consists in to see for any  block of Type-II how many blocks of type $\Gamma_1$ are necessary so that via a ($\tau,\upsilon)-$invariant transformation we reduce the size of the block and obtain some $H_2(-1)$ block. 

Finally, Lemma \ref{joining-Hs_lem} is the counterpart of Lemmas \ref{joining-Js_lem} and \ref{joining-Gammas_lem} for Type-II blocks. Again, instead of the blocks $H_{2k}(\mu)$ we use the tridiagonal version, $\widetilde H_{2k}(\mu)$. In the statement, $\widetilde H_0(\mu)$ stands for an empty block.

\begin{lemma}\label{joining-Hs_lem} The following  transformations are  ($\tau,\upsilon)-$invariant:
\begin{enumerate}[{\rm(i)}]
    \item $\widetilde{H}_{2k}(\mu) \oplus H_2(-1)\rightsquigarrow \widetilde{H}_{2k-2}(\mu) \oplus  \Gamma_1^{\oplus 2}$,  for $\mu\neq \pm 1$ and $k\geq  1$.
    
    \item $\widetilde{H}_{4k+2}(-1) \oplus H_2(-1)\rightsquigarrow\widetilde\Gamma_{2k}^{\oplus2}\oplus\Gamma_1^{\oplus2} $, for $k\geq 1$. 
    
    \item $\widetilde{H}_{4k}(1) \oplus H_2(-1) \rightsquigarrow\widetilde\Gamma_{2k-1}^{\oplus2}\oplus\Gamma_1^{\oplus2}$, for $k\geq 1$. 

\end{enumerate}
\end{lemma}

\begin{proof} In order to see that all transformations in (i)--(iii) are $(\tau,\upsilon)-$invariant, first note that
$$
\begin{array}{lclclc}
    \tau\big(\widetilde{H}_{2k}(\mu) \oplus H_2(-1)\big)&=&k+2&=&\tau\big(\widetilde{H}_{2k-2}(\mu) \oplus  \Gamma_1^{\oplus 2}\big),&\mbox{for $\mu\neq\pm1$ and $k\geq1$}\\ 
    \upsilon\big(\widetilde{H}_{2k}(\mu) \oplus H_2(-1)\big)&=&2k&=&\upsilon\big(\widetilde{H}_{2k-2}(\mu) \oplus  \Gamma_1^{\oplus 2}\big),&\mbox{for $\mu\neq\pm1$ and $k\geq1$}\\
        \tau\big(\widetilde{H}_{4k+2}(-1) \oplus H_2(-1)\big) &= &2k+2&=&\tau\big(\widetilde\Gamma_{2k}^{\oplus2}\oplus \Gamma_1^{\oplus 2}\big),&\mbox{for $k\geq 1$},\\
        \upsilon\big(\widetilde{H}_{4k+2}(-1) \oplus H_2(-1)\big) &=&4k&=&\upsilon\big(\widetilde\Gamma_{2k}^{\oplus2}\oplus \Gamma_1^{\oplus 2}\big),&\mbox{for $k\geq 1$},\\
        \tau\big(\widetilde{H}_{4k}(1) \oplus H_2(-1)\big)&=&2k+2&=&\tau\big(\widetilde\Gamma_{2k-1}^{\oplus2} \oplus \Gamma_1^{\oplus 2}\big),&\mbox{for $k\geq1$},\\
        \upsilon\big(\widetilde{H}_{4k}(1) \oplus H_2(-1)\big)&=&4k&=&\upsilon\big(\widetilde\Gamma_{2k-1}^{\oplus2} \oplus \Gamma_1^{\oplus 2}\big),&\mbox{for $k\geq1$}.
\end{array}
$$
Now let us prove the consistence in (i)--(iii). The following identity is used:
$$
\widetilde H_{2k}(\mu)=\begin{bmatrix}
\widetilde H_{2k-2t}(\mu)&E_{(2k-2t)\times2t}\\0&\widetilde H_{2t}(\mu)
\end{bmatrix},\qquad\mbox{for $t<k$.}
$$

\begin{enumerate}[(i)] 

\item If $k=1$ then $\widetilde{H}_2(\mu) \oplus H_2(-1) \overset{X_1}{\rightsquigarrow} \Gamma_1^{\oplus 2}$, for $X_1=\begin{smat} 
 1 & \mathfrak i \\
 \frac{1}{1+\mu } & -\frac{\mathfrak i}{1+\mu } \\
 0 & 1-\mu  \\
 -\frac{\mathfrak i}{1+\mu } & 0 
\end{smat}$.

If $k=2$ then $\widetilde{H}_{4}(\mu) \oplus H_2(-1) \overset{X_2}{\rightsquigarrow} \widetilde{H}_{2}(\mu) \oplus \Gamma_1^{\oplus 2}$, for  
$X_2={\scriptsize\left[\begin{array}{c|c}
\begin{matrix} 1 & \;\;\;\;\;\;0\;\;\;\;\;\; \\ 0 & 1 \end{matrix} & 
\begin{matrix} 0 & 0 \\ 0 & 0 \end{matrix} \\ \hline
\begin{matrix}
0 & 0 \\ 0 & -\frac{1}{1+\mu } \\ 
0 & -\mathfrak i \\ 0 & -\frac{\mathfrak i}{-1+\mu^2}
\end{matrix} & \text{\normalsize $X_1$}
\end{array}\right]}.$

%$X_2= {\scriptsize\left[\begin{array}{cc|c}1 & 0 & 0\;\;\; 0  \\ 0 & 1 & 0\;\;\; 0  \\ \hline 0 & 0 &   \\
% 0 & -\frac{1}{1+\mu } &   \\ 0& -\mathfrak i & \text{\normalsize $X_1$} \\0 & -\frac{\mathfrak i}{-1+\mu^2} &\end{array}\right]}$.

%\begin{smat}
% 1 & 0 & 0 & 0 \\
% 0 & 1 & 0 & 0 \\
% 0 & 0 & 1 & \mathfrak i \\
% 0 & -\frac{1}{1+\mu } & \frac{1}{1+\mu } & -\frac{\mathfrak %i}{1+\mu } \\
% 0 & -\mathfrak i & 0 & 1-\mu  \\
% 0 & -\frac{\mathfrak i}{-1+\mu^2} & -\frac{\mathfrak i}{1+\mu %} & 0
%\end{smat}$. 

If $k>2$ then $\widetilde{H}_{2k}(\mu) \oplus H_2(-1) \overset{X_k}{\rightsquigarrow} \widetilde{H}_{2k-2}(\mu) \oplus \Gamma_1^{\oplus 2}$, for $X_k=I_{2k-4}\oplus X_2$, since 

\begin{align*} \label{HtheMainPattern}
(I_{2k-4}\oplus X_2)^\top \Big(\widetilde{H}_{2k}(\mu) \oplus H_2(-1)\Big) (I_{2k-4}\oplus X_2) &=\begin{bmatrix}
I_{2k-4} & 0 \\
0 & X_2^\top
\end{bmatrix}
\begin{bmatrix}
\widetilde{H}_{2k-4}(\mu) & E_{(2k-4)\times 6} \\
0 & \widetilde{H}_4(\mu) \oplus H_2(-1)
\end{bmatrix}
\begin{bmatrix}
I_{2k-4} & 0 \\
0 & X_2
\end{bmatrix}\\
&= 
\begin{bmatrix}
\widetilde{H}_{2k-4}(\mu) & E_{(2k-4)\times 6} X_2 \\
0 & X_2^\top \left(\widetilde{H}_4(\mu) \oplus H_2(-1)\right) X_2
\end{bmatrix} \\
&= \begin{bmatrix}
\widetilde{H}_{2k-4}(\mu) & E_{(2k-4)\times 4} \\
0 & \widetilde{H}_2(\mu)\oplus \Gamma_1^{\oplus 2}
\end{bmatrix}\\
\nonumber &= \widetilde{H}_{2k-2}(\mu)\oplus \Gamma_1^{\oplus 2}.
\end{align*}

%%%%%%%%%%%%%%%%%%%%%%%%%%%%%%%%%%%%%%%%%%%%%%%%%%

\item 
Let us prove, for $k\geq1$, that 
%\begin{equation} \label{H2t+G12}
 %   \widetilde{H}_{2t}(-1) \oplusH_2(-1) \rightsquigarrow \widetilde {H}_{2t-2}(-1) \oplus \Gamma_1^{\oplus 2}, \text{ for } t\geq 2.
%\end{equation} 
%since the result follows after repeating  this argument $2g$ times starting with $2t=4k+2$. We should note that $\widetilde H_{4h}(-1)$ is not congruent to a Type-II block as in Theorem~\ref{cfc_th} (It can be showed that  $\widetilde H_{4h}(-1)$ is congruent to $\Gamma_{2h}^{\oplus 2}$), but they are well defined and  used as intermediate steps.  
$$\widetilde{H}_{4k+2}(-1) \oplus H_2(-1)\overset{X_k}{\rightsquigarrow} \widetilde{H}_{4k}(-1) \oplus \Gamma_1^{\oplus 2}, \text{ for } X_k=I_{4k-2}\oplus C \text{, with } C=\begin{smat} 
 1 & 0 & 0 & 0 \\
 0 & 0 & 1 & -\mathfrak i \\
 0 & 0 & 1 & \mathfrak i \\
 0 & 0 & 0 & -\mathfrak i \\
 0 & -1 & 1 & -\mathfrak i \\
 1 & 0 & 0 & 0 \\
\end{smat}.$$ 
This is because 
\begin{align*} 
(I_{4k-2}\oplus C)^\top \Big(\widetilde{H}_{4k+2}(-1) \oplus H_2(-1)\Big) (I_{4k-2}\oplus C)  & =
\begin{bmatrix}
I_{4k-2} & 0 \\
0 & C^\top
\end{bmatrix}
\begin{bmatrix}
\widetilde{H}_{4k-2}(-1) & E_{(4k-2)\times 6} \\
0 & \widetilde{H}_4(-1) \oplus H_2(-1) 
\end{bmatrix}
\begin{bmatrix}
I_{4k-2} & 0 \\
0 & C
\end{bmatrix}
\\ & = 
\begin{bmatrix}
\widetilde{H}_{4k-2}(-1) & E_{(4k-2)\times 6} C \\
0 & C^\top \left(\widetilde{H}_4(-1) \oplus H_2(-1)\right) C
\end{bmatrix} \\
&= \begin{bmatrix}
\widetilde{H}_{4k-2}(-1) & E_{(4k-2)\times 4} \\
0 & \widetilde{H}_2(-1)\oplus \Gamma_1^{\oplus 2} 
\end{bmatrix}\\
\nonumber &= \widetilde{H}_{4k}(-1)\oplus \Gamma_1^{\oplus 2}.
\end{align*}
%where, in the last-but-one identity, we have used that $\widetilde{H}_{4}(-1) \oplus \widetilde{H}_{2}(-1) \overset{C}{\rightsquigarrow} \widetilde{H}_{2}(-1)\oplus \Gamma_1^{\oplus 2}$.

Finally, let us see that $\widetilde H_{4k}(-1)$ is congruent to $\widetilde\Gamma_{2k}^{\oplus2}$ or, equivalently, that $H_{4k}(-1)$ is congruent to $\Gamma_{2k}^{\oplus2}$.
%$\widetilde H_{2t-2}(-1)=\widetilde H_{4k}(-1)$, which is congruent to $H_{4k}(-1)$, and this in turn is congruent to $\Gamma_{2k}^{\oplus2}$. 
%Since $\widetilde H_{4k}(-1)$ is congruent to $H_{4k}(-1)$ and $\Gamma_{2k}^{\oplus2}$ is congruent to $\widetilde\Gamma_{2k}^{\oplus2}$, it suffices to prove that $H_{4k}(-1)$ is congruent to $\Gamma_{2k}^{\oplus2}$. 
In order to do this, we are going to prove that the cosquares of $H_{4k}(-1)$ and $\Gamma_{2k}^{\oplus2}$ %, namely $H_{4k}(-1)^{-\top}H_{4k}(-1)$ and $\left(\Gamma_{2k}^{\oplus2}\right)^{-\top}\Gamma_{2k}^{\oplus2}$,
 are similar, and this immediately implies that $H_{4k}(-1)$ and $\Gamma_{2k}^{\oplus2}$ are congruent, by Lemma \ref{cosquares_lem}.

 The cosquare of $H_{4k}(-1)$ is 
$$
\begin{array}{ccl}
H_{4k}(-1)^{-\top} H_{4k}(-1)&=&\begin{bmatrix}
0&I_{2k}\\J_{2k}(-1)&0
\end{bmatrix}^{-\top}\begin{bmatrix}
0&I_{2k}\\J_{2k}(-1)&0
\end{bmatrix}\\&=&\begin{bmatrix}
0&J_{2k}(-1)^{-1}\\I_{2k}&0
\end{bmatrix}^\top\begin{bmatrix}
0&I_{2k}\\J_{2k}(-1)&0
\end{bmatrix}\\&=&\begin{bmatrix}
0&I_{2k}\\J_{2k}(-1)^{-\top}&0
\end{bmatrix}\begin{bmatrix}
0&I_{2k}\\J_{2k}(-1)&0
\end{bmatrix}\\&=&\begin{bmatrix}
J_{2k}(-1)&0\\0&J_{2k}(-1)^{-\top}
\end{bmatrix},
\end{array}
$$
and the cosquare of $\Gamma_{2k}^{\oplus2}$ is
$$
\begin{bmatrix}
\Gamma_{2k}^{-\top}&0\\0&\Gamma_{2k}^{-\top}
\end{bmatrix}\begin{bmatrix}
\Gamma_{2k}&0\\0&\Gamma_{2k}
\end{bmatrix}=\begin{bmatrix}
\Gamma_{2k}^{-\top}\Gamma_{2k}&0\\0&\Gamma_{2k}^{-\top}\Gamma_{2k}
\end{bmatrix},
$$
with (see \cite[p. 13]{semaj})
$$
\Gamma_{2k}^{-\top}\Gamma_{2k}=\begin{bmatrix}
-1&-2&&\star\\&\ddots&\ddots&\\&&-1&-2\\0&&&-1
\end{bmatrix},
$$
%where $\star$ denotes some entries that are not relevant in our arguments (see \cite[p. 13]{semaj}). Since $J_{2k}(-1)^{-\top}$  is similar to $J_{2k}(-1)$, the Jordan canonical form of both $H_{4k}(-1)^{-\top}H_{4k}(-1)$ and $\left(\Gamma_{2k}^{\oplus 2}\right)^{-\top}\Gamma_{2k}^{\oplus2}$ consist of two Jordan blocks of size $2k$ associated with the eigenvalue $-1$ and, as a consequence, $H_{4k}(-1)^{-\top}H_{4k}(-1)$ and $\left(\Gamma_{2k}^{\oplus 2}\right)^{-\top}\Gamma_{2k}^{\oplus2}$ are similar, as claimed.
 where $\star$ denotes some entries that are not relevant in our arguments. As $J_{2k}(-1)^{-\top}$  is similar to $J_{2k}(-1)$, the previous identities show that  $\big(H_{4k}(-1)\big)^{-\top}H_{4k}(-1)$ and $\left(\Gamma_{2k}^{\oplus 2}\right)^{-\top}\Gamma_{2k}^{\oplus2}$ are similar, since the Jordan canonical form of both them is $J_{2k}(-1)^{\oplus 2}$.
 
%Now, the result follows because $\Gamma_{2k}$ is congruent to $\widetilde\Gamma_{2k}$.
%%%%%%%%%%%%%%%%%%%%%%%%%%%%%%%%%%%%%%%%%%%%%%%%%%

\item Let us prove that, for $k\geq1$:
\begin{eqnarray*} 
\widetilde{H}_{4k}(1) \oplus H_2(-1) \overset{X_k}{\rightsquigarrow}  \widetilde{H}_{4k-2}(1) \oplus \Gamma_1^{\oplus 2},\quad\text{for $X_k=I_{4k-4}\oplus C$,\  with $C=\begin{smat} 
 1 & 0 & 0 & 0 \\
 0 & 1 & 0 & 0 \\
 0 & 0 & 1 & \mathfrak i \\
 0 & -\frac{1}{2} & \frac{1}{2} & -\frac{\mathfrak i}{2} \\
 0 & 0 & 1 & \mathfrak i \\
 0 & \frac{1}{2} & 0 & 0 
\end{smat}$.}
%\label{H2k+2G1^2->H2kH2-1}
\end{eqnarray*}
For $k=1$ the solution matrix is $X_1=C$, as it can directly checked. Let us now see it for $k\geq2$:
\begin{align*}
(I_{4k-4}\oplus C)^\top \left(\widetilde{H}_{4k}(1) \oplus H_2(-1)\right) (I_{4k-4}\oplus C)
& = \begin{bmatrix}I_{4k-4}&0\\0&C^\top\end{bmatrix}\begin{bmatrix}
\widetilde H_{4k-4}(1)& E_{(4k-4)\times 6} \\0& \widetilde H_{4} (1) \oplus H_2(-1)
\end{bmatrix}\begin{bmatrix}I_{4k-4}&0\\0&C\end{bmatrix}\\
&=
\begin{bmatrix}
\widetilde{H}_{4k-4}(1) & E_{(4k-4)\times 6} C \\
0 & C^\top \left(\widetilde{H}_4(1) \oplus H_2(-1)\right) C
\end{bmatrix}
\\ &= \begin{bmatrix}
\widetilde{H}_{4k-4}(1) & E_{(4k-4)\times 4} \\
0 & \widetilde{H}_2(1)\oplus \Gamma_1^{\oplus 2} 
\end{bmatrix}\\
&= \widetilde{H}_{4k-2}(1)\oplus \Gamma_1^{\oplus 2}.
\end{align*}
It remains to see that $\widetilde H_{4k-2}(1)$ is congruent to $\widetilde\Gamma_{2k-1}^{\oplus2}$ or, equivalently, that $H_{4k-2}(1)$ is congruent to $\Gamma_{2k-1}^{\oplus2}$.
To prove this, we can proceed as before, by showing that the cosquares of $ H_{4k-2}(1)$ and $\Gamma_{2k-1}^{\oplus2}$ are similar (in this case, their Jordan canonical form is $J_{2k-1}(1)^{\oplus2}$), and this implies that $ H_{4k-2}(1)$ and $\Gamma_{2k-1}^{\oplus2}$ are congruent, again by Lemma \ref{cosquares_lem}. 
\end{enumerate}
\end{proof}

\section{The main result}\label{main_sec}

The following result, which is the main result in this work, improves the main result in \cite{bcd} (namely, Theorem 8 in that reference) by including the case where the CFC of $A$ contains blocks of type $H_2(-1)$, that were excluded in \cite[Th. 8]{bcd}.

\begin{theorem}\label{main_th}
Let $A$ be a complex square matrix whose {\rm CFC} does not have blocks of type $\widetilde H_4(1)$, and $B$ a symmetric matrix. Then $X^\top AX= B$ is consistent if and only if $ \rank B\leq\min\{\tau(A),\upsilon(A)\}$. 
\end{theorem}
\begin{proof} The necessity of the condition is already stated in Theorem \ref{necessary_th}. We are going to prove that it is also sufficient.

By the $J_1(0)$-law and the Canonical reduction law, we may assume that both $A$ and $B$ are given in CFC and that neither $A$ nor $B$ have blocks of type $J_1(0)$. This implies, in particular, that $B=I_m$, for some $m$, and that $A$ is as in Definition \ref{components_CFC_def}, with $j_1=0$. We also assume that all blocks $\Gamma_k$ and $H_{2k}(\mu)$ in $A$, if present, have been replaced by $\widetilde\Gamma_k$ and $\widetilde H_{2k}(\mu)$, respectively.

Let us recall that $\Gamma_1^{\oplus m}=I_m$. Throughout the proof, we mainly use the first notation, to emphasize that we are dealing with canonical blocks. 

If the CFC of $A$ does not contain blocks $H_2(-1)$, then the result is provided in \cite[Th. 8]{bcd}. Otherwise, we are going to see that it is possible, by means of $(\tau,\upsilon)-$invariant transformations, to either ``absorb" all blocks $H_2(-1)$ or to end up with a direct sum of blocks $H_2(-1)$, together with, possibly, other blocks, which are quite specific. More precisely, we can end up with a direct sum of blocks satisfying one of the following conditions:
\begin{enumerate}[\rm (C1)]\setcounter{enumi}{-1}
\item\label{case0} There are no blocks $H_2(-1)$.
    %\item\label{case1} There is exactly one block $H_2(-1)$, together with, possibly, a block $\widetilde H_8(1)$ and/or a direct sum of blocks $\Gamma_1$ and/or $\Gamma_2$.
   % \item\label{case2} There are exactly two blocks $H_2(-1)$ together with, possibly, a block $H_8(1)$ and a direct sum of blocks $\Gamma_1$ and/or $\Gamma_2$.
    \item\label{case2} There are some blocks $H_2(-1)$ together with, possibly, a direct sum of blocks $J_3(0)$, {$\widetilde\Gamma_2$,} and/or $\Gamma_1$.
\end{enumerate}
We are first going to see that, indeed, we can arrive to one of the situations described in cases (C\ref{case0})--(C\ref{case2}). { In the procedure, we may need to permute the canonical blocks, in order to use Lemmas \ref{joining-Js_lem}, \ref{joining-Gammas_lem}, and \ref{joining-Hs_lem}. By Theorem \ref{cfc_th}, this provides a congruent matrix which has, in particular, the same $\tau$ and $\upsilon$, so these permutations do not affect the consistency.} Then, we will prove that in both cases (C\ref{case0}) and (C\ref{case2}) the statement holds.

So let us assume that the CFC of $A$ contains a direct sum of blocks $H_2(-1)$, together with some other Type-0, Type-I, and Type-II blocks (except $\widetilde H_4(1)$). 

Using Lemma \ref{joining-Js_lem}, for each block $J_k(0)$ (with $k\neq3$) we can ``absorb" a block $H_2(-1)$ by means of a $(\tau,\upsilon)-$invariant transformation, and we end up with a direct sum of a block $J_{k-2}(0)$ together with two blocks $\Gamma_1$. We can keep reducing the size of the Type-0 blocks until either all $H_2(-1)$ blocks have been absorbed (so we end up in case (C\ref{case0})) or there are no more Type-0 blocks, except maybe blocks $J_3(0)$. Now, we can proceed in the same way with Type-I blocks using Lemma \ref{joining-Gammas_lem}. Again, we end up either with a direct sum containing no $H_2(-1)$ blocks (case (C\ref{case0}) again) or no Type-I blocks, except maybe blocks $\Gamma_1$ and/or $\widetilde\Gamma_2$. Next, we do the same with Type-II blocks using Lemma \ref{joining-Hs_lem}. Note that the reductions in parts (ii) and (iii) in the statement of Lemma \ref{joining-Hs_lem} produce as an output some Type-I blocks $\widetilde\Gamma_{k}$, with $k\geq1$. In the case when $k>1$, we can use again Lemma \ref{joining-Gammas_lem}, provided that there are still blocks $H_2(-1)$. 
Therefore, after these reductions, either we have absorbed all blocks $H_2(-1)$ (case (C\ref{case0}) again), or there are blocks $H_2(-1)$, together with, possibly, a direct sum of other blocks that cannot absorb them, namely $J_3(0)$, $\widetilde\Gamma_2$, and/or $\Gamma_1$ (case (C\ref{case2})).

Now, it remains to prove that in both cases (C\ref{case0}) and (C\ref{case2}) the statement holds, namely that $A\rightsquigarrow \Gamma_1^{\oplus m}$, for any $m\leq\min\{\tau( A),\upsilon(A)\}$, in these two cases. Let $\widehat A$ be the matrix obtained after applying to $A$ all the transformations explained in the previous paragraph. By the Transitive law, $A\rightsquigarrow\widehat A$. Moreover, since all these transformations are $(\tau,\upsilon)-$invariant, then \eqref{additive} implies that $\tau(A)=\tau(\widehat A)$ and $\upsilon(A)=\upsilon(\widehat A)$. Therefore, it is enough to prove that $\widehat A\rightsquigarrow \Gamma_1^{\oplus m}$ for any $m\leq\min\{\tau(\widehat A),\upsilon(\widehat A)\}$.  By the Elimination law, $\Gamma_1^{\oplus a} \rightsquigarrow \Gamma_1^{\oplus b}$ for any $b<a$, so it will be enough to prove that  $\widehat A\rightsquigarrow \Gamma_1^{\min\{\tau(\widehat A),\upsilon(\widehat A)\}}$.

In case (C\ref{case0}) the statement is true, as a consequence of \cite[Th. 8]{bcd}. More precisely, in this case, $\min\{\tau(A),\upsilon(A)\}=\tau(A)$, as a consequence of Lemma \ref{rho<=sigma_lem}. Then, \cite[Th. 8]{bcd} guarantees that $A\rightsquigarrow \Gamma_1^{\oplus\tau(A)}$ (in \cite[Th. 8]{bcd}, however, the notation $\tau$ was not used).

In case (C\ref{case2}), we may assume that $$\widehat A=H_2(-1)^{\oplus j}\oplus J_3(0)^{\oplus h}\oplus  \widetilde\Gamma_2^{\oplus k}\oplus\Gamma_1^{\oplus\ell} \quad \text{for some $j,h,k,\ell\geq0$}.$$ 
Note that, in this case, $\min\{\tau(\widehat A),\upsilon(\widehat A)\}=\upsilon(\widehat A)$, since 
$\tau(\widehat A)=j+2h+k+\ell > \upsilon(\widehat A)=2h+k+\ell$.
Hence, it is enough to prove that  $A\rightsquigarrow \Gamma_1^{\upsilon(\widehat A)}$. In order to do this, we consider the transformations
$$
H_2(-1)^{\oplus j}\oplus J_3(0)^{\oplus h}\oplus  \widetilde\Gamma_2^{\oplus k}\oplus\Gamma_1^{\oplus\ell} 
\rightsquigarrow
J_3(0)^{\oplus h}\oplus  \widetilde\Gamma_2^{\oplus k}\oplus\Gamma_1^{\oplus\ell} 
%\rightsquigarrow
%\Gamma_1^{\oplus 2h}\oplus \widetilde\Gamma_2^{\oplus k}\oplus\Gamma_1^{\oplus\ell}
\rightsquigarrow
\Gamma_1^{\oplus 2h}\oplus \Gamma_1^{\oplus k}\oplus\Gamma_1^{\oplus\ell}=\Gamma_1^{\oplus \upsilon (\widehat A)},
$$
where the first transformation is a consequence of the Elimination law, and the second transformation is a consequence of the Addition law, together with Lemma \ref{j3_lem} (for the first addend) and with $\widetilde\Gamma_2 \overset{\begin{smat} 1 \\ 0\end{smat}}{\rightsquigarrow} \Gamma_1$ (for the second addend).
\end{proof}

\begin{remark}
Unfortunately, when the CFC of $A$ contains at least one block $\widetilde H_4(1)$, it is no longer true that, for any $m\leq\min\{\tau(A),\upsilon(A)\}$, the equation $X^\top AX=I_m$ is consistent. For instance, $X^\top \widetilde H_4(1)X=I_3$ is not consistent (see \cite[Th. 7]{bcd}), but $\tau(\widetilde H_4(1))=4$ and $\upsilon(\widetilde H_4(1))=3$, so $\min\{\tau(\widetilde H_4(1)),\upsilon(\widetilde H_4(1))\}=3$. Therefore, the case where the CFC of $A$ contains blocks $\widetilde H_4(1)$ deserves a further analysis.

 Related to this, Theorem \ref{main_th} can be slightly improved, allowing the CFC of $A$ to contain blocks $\widetilde H_4(1)$ provided that the number of these blocks is not larger than the number of blocks $H_2(-1)$. In this case, we can start the reduction procedure described in the proof of Theorem \ref{main_th} by ``absorbing" the blocks $\widetilde H_4(1)$ with the blocks $H_2(-1)$ as described in Lemma \ref{joining-Hs_lem}-(iii). More precisely, we can gather each block $\widetilde H_4(1)$ with a block $H_2(-1)$, and use the $(\tau,\upsilon)-$invariant transformation $\widetilde H_4(1)\oplus H_2(-1)\rightsquigarrow \Gamma_1^{\oplus4}$. Once we have absorbed all blocks $\widetilde H_4(1)$ we can continue with the reduction as explained in the proof of Theorem \ref{main_th}.
\end{remark}

 \section{Conclusions and open questions}\label{conclusion_sec}
In this paper, we have obtained a necessary condition for the equation $X^\top AX=B$ to be consistent, with $A,B$ being complex square matrices and $B$ being symmetric. This condition improves the one obtained in \cite[Th. 2]{bcd}. Moreover, we have proved that the condition is sufficient when the CFC of $A$ does not contain blocks $\widetilde H_4(1)$. This result also improves the one in \cite[Th. 8]{bcd}, where the case in which the CFC has blocks $H_2(-1)$ was excluded.

As a natural continuation of this work it remains to address the case where the CFC of $A$ contains blocks $\widetilde H_4(1)$, in order to fully characterize the consistency of $X^\top AX=B$, with $B$ symmetric, for any matrix $A$. We have seen that the condition mentioned above is no longer sufficient in this case, so a different characterization is needed. So far, we have been unable to find such a characterization.

\bigskip

\noindent{\bf Acknowledgments:} This research has been funded by the {\em Agencia Estatal de Investigaci\'on} of Spain through grants PID2019-106362GB-I00/AEI/10.13039/501100011033 and MTM2017-90682-REDT.

\end{document}